\documentclass[a4paper, 12pt]{article}
\usepackage{amsfonts}
\usepackage{fullpage}
\usepackage[T1]{fontenc}
\usepackage[latin2]{inputenc}
\usepackage{amsmath}
\usepackage{amsthm}
\usepackage{amssymb}
\usepackage{bbm}
\usepackage{graphicx}

\numberwithin{equation}{section}
\newtheorem{thm}{Theorem}[section]
\newtheorem{lemma}[thm]{Lemma}
\newtheorem{prop}[thm]{Proposition}

\theoremstyle{definition}
\newtheorem{defi}{Definition}[section]
\newcommand{\e}{\epsilon}
\newcommand{\R}{\mathbb{R}}
\newcommand{\Z}{\mathbb{Z}}
\newcommand{\p}{\mathbb{P}}
\begin{document}

\title{Quenched Invariance Principle for the Random Walk on the Penrose
Tiling }
\author{Zs. Bartha \footnote{Zsolt Bartha, Budapest University of Technology
and Economics, Magyar tudosok korutja 2, \ 1117
Budapest HUNGARY, barthazs@math.bme.hu}
\footnote{Supported in part by OTKA (Hungarian
Scientific Research Fund) Grant K109684.}
\and A. Telcs \footnote{Andras Telcs, Prof. of Quantitative Methods, Department of Quantitative
Methods, Faculty of Economics, Univ. Pannonia, 8200 Veszprém
HUNGARY \& Assoc. Prof. Department of Computer Science and Information
Theory, Budapest University of Technology and Economics, Magyar tudosok korutja 2, \ 1117
Budapest HUNGARY, telcs.szit.bme@gmail.com}}
\date{July 4, 2014}
\maketitle

\begin{abstract}
We consider the simple random walk on the graph corresponding to a Penrose
tiling. We prove that the path distribution of the walk converges weakly to
that of a non-degenerate Brownian motion for almost every Penrose tiling
with respect to the appropriate invariant measure on the set of tilings. Our
tool for this is the corrector method.

\vspace{0.3 cm}
\end{abstract}

\section{Introduction}

The Penrose tiling is the most famous aperiodic tiling of the plane, that
is, a covering with given polygons (two kinds of rhombuses in this case)
without overlaps or gaps which no translation of the plane maps to itself.
We already have a good understanding of the structure of the Penrose
tilings, and several ways are known to construct them. Because of these
properties, they are studied in the field of diffusion in quasicrystalline
environment. In \cite{sz} Domokos Szász introduced that in the Lorentz gas
model it would be worth studying the case (as an aperiodic one) when the
billiard obstacles are placed corresponding to the Penrose tiling. He
conjectured that the scaling limit of the diffusion is a Brownian motion. A
closely related model to this is the simple random walk on the rhombuses of
the Penrose tiling, precisely, which runs on the graph with the centers of
the rhombuses as vertices and the edges spanned by centers of neighboring
rhombuses; we call this a Penrose graph. Considering the set of Penrose
tilings of the plane with the appropriate shift invariant measure,
we actually examine a random walk in random environment.
The question of our interest is that whether the invariance principle holds,
that is, the scaling limit of the process is a Brownian motion. András Telcs
proved this in \cite{telcs} for the annealed case, that is, for the case
where we take the average of the walk with respect to the mentioned
invariant measure. The result of this paper is that the invariance principle
is also true for the quenched case, that is, for almost every concrete
tiling. More precisely, we firstly consider a fixed Penrose-tiling $\omega_0$
on the plane (we also refer to this as a configuration)
where the origin is the center of a tile, and confine ourselves to the set
$\Omega$ of tilings which can be obtained from $\omega_0$ by shifting the
tiling in such a way that the center of a tile moves to the origin.
On the set of tilings in $\Omega$ we can define a shift-invariant
probability measure with respect to which a step along a fixed principal
direction is ergodic \cite{r}. Denote this measure by $\mathbb{P}$ and the
corresponding expectation by $\mathbb{E}$. Using the notation $\mathcal{W}_T$
for the usual Borel-$\sigma$-algebra on $C[0,T]$, the space of continuous
functions on $[0,T]$, our result is the following:

\begin{thm}\label{main}
Let $(X_n)_{n\ge 0}$ be the simple symmetric random walk on the Penrose graph
starting from the origin and let 
\begin{equation*}
\tilde B_n(t)=\frac{1}{\sqrt{n}}(X_{\lfloor tn\rfloor}+(tn-\lfloor
tn\rfloor)(X_{\lfloor tn\rfloor+1}-X_{\lfloor tn\rfloor}))\quad t\ge 0.
\end{equation*}
Then for all $T>0$ and for $\mathbb{P}$-almost every $\omega\in\Omega$, the
law of $(\tilde B_n(t):0\le t\le T)$ on $(C[0,T],\mathcal{W}_T)$ converges
weakly to the law of a non-degenerate Brownian motion $(B_t:0\le t\le T)$.
\end{thm}

Our proof applies the corrector method, a common way of proving
invariance principles, firstly used by Kipnis and Varadhan \cite{kv}.
It was also used by Berger and Biskup in \cite{bb} and by Biskup and
Prescott in \cite{bp} to prove the invariance principle on the supercritical
percolation cluster of $\mathbb{Z}^2$. The corrector method means the
modification of the graph with the shift of the vertices such that the walk
on the new graph is a martingale; for this case strong theorems are
applicable. The main tasks are therefore the proof of the existence of the
corrector, and the estimation of the bias caused by that with respect to the
scaling limit. For the former, spectral theoretical considerations are
needed, while the latter is based on the already proven ergodicity of the
walk.

Our paper is organized as follows: in Section 2, we give the definition
of the Penrose tiling, and a construction that our proof will rely on.
This will show us the structure of the set of all Penrose tilings and
the shift-invariant measure on it. In Section 3, we prove the existence
of the corrector and its sublinearity. We conclude from this
the invariance principle for the corrected Penrose graph in Section 4,
and for the original Penrose graph in Section 5.
Throughout the paper we follow some of the main steps of \cite{bb}.

\section{Basics of the construction of Penrose tilings and the invariant measure}

Penrose tilings are builded from two kinds of rhombuses with certain
matchings rules (see for example Penrose's original paper \cite{penrose}).
Figure 1 shows these rhombuses. The "thin" rhombus has angles of $\pi/5$
and $4\pi/5$, while the "thick" rhombus has angles of $2\pi/5$ and $3\pi/5$.
A Penrose tiling is any tiling of the plane without gaps and overlaps
using only these two kinds of rhombuses that obeys the matching rules
also shown in Figure 1: two rhombuses can be placed next to each other
only in a way such that their common side has the same type of arrow
pointing in the same direction. Penrose showed that any of these
tilings is aperiodic, i.e., there is no translation of the plane that
maps a tiling to itself.

\begin{figure}[h]
\begin{center}
\caption{The two kinds of rhombuses with the matching rules}
\includegraphics[width=14 cm]{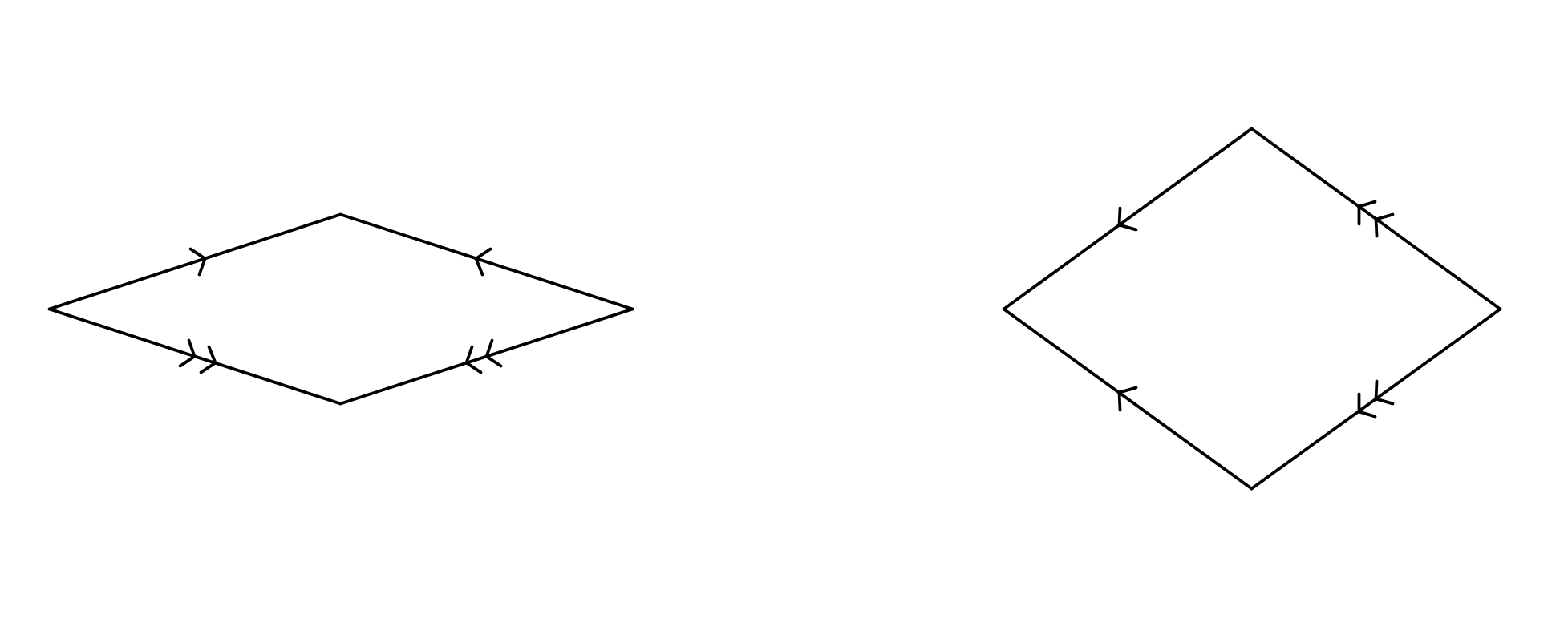}
\end{center}
\end{figure}

The results of de Bruijn \cite{bruijn} show a property of the Penrose tiling
(or from another point of view, an equivalent definition) that will play a
crucial role in our proof. His idea of relating Penrose tilings
to pentagrids was the basis of the work of Kunz \cite{kunz} who gave the invariant,
ergodic measure that our proof relies on.

To start, we define the vectors
\begin{align*}
{\bf e}_k&=\left(\cos\frac{2k\pi}{5},\sin\frac{2k\pi}{5}\right),\\
{\bf e}_{k\bot}&=\left(\cos\left(\frac{\pi}{2}+\frac{2k\pi}{5}\right),\sin\left(\frac{\pi}{2}+\frac{2k\pi}{5}\right)\right),
\end{align*}
for $k=0,1,2,3,4$, and the grids
\[G_k=\{{\bf x}\in\R^2:{\bf x}\cdot{\bf e}_{k\bot}-\gamma_k\in\Z\},\]
where $\gamma_k\in[0,1)$ and
\[\sum_{k=0}^4\gamma_k=0.\]
Thus, the grid $G_k$ is a set of parallel lines oriented along ${\bf e}_k$.
We call the union of $G_0,\ldots,G_4$ a pentagrid.
Denote by $\mathcal{S}$ the set of intersections of any two lines of the
pentagrid. Given such a pentagrid, the corresponding Penrose
tiling is constructed by associating with each $x\in\mathcal{S}$ a tile
containing $x$ with edges perpendicular to the lines intersecting at $x$.
This is a thick rhombus if $x\in G_i\cap G_j$ with $|j-i|\in\{1,4\}$ and a thin
rhombus if $|j-i|\in\{2,3\}$.
(This construction works when there is no point of the plane that belongs
to at least three lines of the pentagrid, otherwise the tiling can be constructed
as a limit in a particular sense.) By de Bruijn's proof the matching rules are
satisfied in this procedure, and we obtain a Penrose tiling. Furthermore, every Penrose
tiling can be obtained from an appropriate pentagrid. We call the set of tiles
(or the corresponding set of graph vertices) associated with the
intersections on a given line of the pentagrid a ribbon. In other words,
this is a path of consecutively neighboring tiles where the borders between
neighboring tiles are all parallel. When we fix a center of a tile as the
origin of the plane, we will refer to the directions of the lines of the
pentagrid through this origin as principal directions.

From this construction, it follows that the Penrose tilings of the plane,
up to translations, can be parametrized as $(i,j;\gamma_n,\gamma_m)$, where
$i,j\in\{0,1,\ldots,5\}, j-i\equiv 1 \text{ or } 2\ \mod(5)$ and $\gamma_n,\gamma_m\in[0,1)$.
This is interpreted as the following: the origin is at a point of $G_i\cap G_j$,
$\gamma_i$ and $\gamma_j$ are chosen to be $0$,
and with $\{l,m,n\}=\{1,\ldots,5\}\setminus\{i,j\}$
we choose $m$ and $n$ so that $m-n\equiv 1 \text{ or } 2\ \mod(5)$ and
the angles $({\bf e}_n,{\bf e}_m)$ and $({\bf e}_i,{\bf e}_j)$ intersect
each other. Moreover, we set $\gamma_l=1-\gamma_n-\gamma_m$. With all of this,
we uniquely determined a pentagrid, and thus a Penrose lattice. This parametrization
shows that $\Omega^*$, the set of all Penrose tilings of the plane is isomorphic to
the union of ten tori.

A (random) walk on a Penrose tiling can be looked at as a sequence of elements of $\Omega^*$:
instead of a moving particle in a fixed environment, we can consider the moving
environment as seen from the particle (the place of which is always the origin by definition).
In this way, a step from one tile to a neighboring one corresponds to a shift of the
Penrose tiling. There are four types of these shifts ($T_1,\ldots,T_4$) corresponding
to the four sides of a tile. (For the way to define $T_1,\ldots,T_4$ unambiguously on the whole
$\Omega^*$, see \cite{kunz}.) Kunz showed that under these shifts the following measure $\mu$ on $\Omega^*$
is invariant and ergodic (i.e., a subset of $\Omega^*$ that is invariant under each $T_i$ has measure
$0$ or $1$). Take the Lebesgue measure on the tori mentioned above (parametrized by the pairs $(i,j)$)
for which $j-i\equiv 2\ \mod(5)$, and $\tau$ times the Lebesgue measure on the tori for which
$j-i\equiv 2\ \mod(5)$, where $\tau=(1+\sqrt{5})/2$. Then normalize this measure so as to
be a probability measure on $\Omega^*$ to obtain $\mu$.

\section{The corrected Penrose graph}

As we mentioned, we would like to find a vector called corrector for each
vertex of the graph for which if we shift the vertices by
the corresponding correctors, leaving the system of the edges unchanged, the
simple symmetric random walk on the new (corrected) graph is a martingale.
In the following we construct this corrector and we prove its sublinear
property which implies that its effect vanishes in the
scaling limit. This allows us to derive the invariance principle for the
original Penrose graph from that of the corrected graph.

\subsection{The construction of the corrector}

First of all, we introduce a few more notations that we use in this paper.
Let $S$ be the set of planar vectors that occur
in configuration $\omega_0$ (or in any $\omega\in\Omega$) as a vector from
the center of a tile to the
center of one of it's neighboring tiles (there are finitely many such
vectors). If in a configuration $\omega\in\Omega$ $x\in\mathbb{R}^2$ is
the center of a tile, then denote by $\tau_x\omega$ the configuration
that we get by shifting $\omega$ so that the
point $x$ moves to the origin. For a given $\omega\in\Omega$ let $%
S_{\omega}\subseteq S$ be the set of the centers of the tiles neighboring
the tile of the origin. Denote by $H\subseteq\mathbb{R}^2$ the set of the
points on the plane that occur in some configuration from $\Omega$ as the
center of a tile. Furthermore, for a given $\omega\in\Omega$ let $%
H_{\omega}\subseteq H$ be the set of points of the plane that occur in
configuration $\omega$ as the center of a tile.

The shift-invariant probability measure $\p$ on $\Omega$, already mentioned in Section 1,
is the restriction of the measure $\mu$, introduced in Section 2,
to $\Omega$, normalized. With a slight abuse of notation
we denote by $L^2=L^2(\Omega)$ the Hilbert space of scalar or vector valued
functions defined on $\Omega$ which are square integrable with respect to $%
\mathbb{P}$, the scalar product being $<f,g>=\mathbb{E}(f\cdot g)$.

For a given $\omega\in\Omega$ let $(X_n)_{n\ge 0}$ be the simple symmetric
random walk on the centers of consecutively neighboring tiles starting from
the origin and let $P_{\omega}$ be the correspoding probability measure. So, 
$P_{\omega}(X_0=0)=1$ and for any $e\in \mathbb{R}^2$ 
\begin{equation}  \label{atmenet}
P_{\omega}(X_{n+1}=x+e|X_n=x)=\frac{1}{4}\mathbbm{1}\{e\in
S_{\omega}\}\circ\tau_x
\end{equation}
Let $Q$ be the Markov operator of the random walk on $\Omega$, so for any
function $f\in L^2(\Omega)$, $Qf\in L^2(\Omega)$ is the following: 
\begin{equation*}
(Qf)(\omega)=\frac{1}{4}\sum_{e\in S_{\omega}}f(\tau_e\omega).
\end{equation*}
Since $Q$ is the Markov operator of a reversible random walk, $%
||Q||_{L^2}\le 1$ and $Q$ is self-adjoint. Denote by $V:\Omega\rightarrow%
\mathbb{R}^2$ the drift at the origin: 
\begin{equation*}
V(\omega)=\frac{1}{4}\sum_{e\in S_{\omega}}e.
\end{equation*}
$V$ is bounded, therefore $V\in L^2(\Omega)$. These imply that the equation 
\begin{equation*}
(1+\epsilon-Q)\psi_\epsilon=V
\end{equation*}
can be solved uniquely for any $\epsilon>0$ with respect to $%
\psi_\epsilon\in L^2$.

\begin{thm}
\label{korrektor} There exists a function $\chi:\Omega\times H\rightarrow%
\mathbb{R}^2$ with which for any $x\in H$: 
\begin{equation*}
\lim_{\epsilon\searrow 0}\mathbbm{1}\{x\in
H_{\omega}\}(\psi_{\epsilon}\circ\tau_x-\psi_{\epsilon})=\chi(x,\cdot)\text{
in $L^2$.}
\end{equation*}
For this function the followings apply:

(1) For $\mathbb{P}$-almost every $\omega\in\Omega$: 
\begin{equation*}
\chi(x,\omega)-\chi(y,\omega)=\chi(x-y,\tau_y\omega)
\end{equation*}
for any $x,y\in H_{\omega}$.

(2) For $\mathbb{P}$-almost every $\omega\in\Omega$ the $x\mapsto\chi(x,%
\omega)+x$ function is harmonic with respect to the transition probabilities
given by \eqref{atmenet}.

(3) There exists a constant $C<\infty$ for which 
\begin{equation*}
||(\chi(x+e,\omega)-\chi(x,\omega))\mathbbm{1}\{x\in H_{\omega}\}\mathbbm{1}%
\{e\in S_{\omega}\}\circ\tau_x||_2<C
\end{equation*}
applies for any $x\in H$ and for any $e\in S$.
\end{thm}

We call the function in the Theorem \ref{korrektor} the corrector. The value
of the corrector at the point $x$ is intuitively the amount 'by which the
drift collected by a walk starting from $x$ is larger than that of a walk
starting from the origin'.

To prove Theorem \ref{korrektor} we examine the $\mu_V$ spectral measure of
the (self-adjoint) operator $Q:L^2\rightarrow L^2$ with respect to $V$, i.
e. the measure for which any continuous, bounded function $%
\phi:[-1,1]\rightarrow\mathbb{R}$ satisfies 
\begin{equation*}
<V,\phi(Q)V>=\int\limits_{-1}^1\phi(\lambda)\mu_V(\mathrm{d}\lambda).
\end{equation*}
(Using that $\mathrm{spec}(Q)\subseteq[-1,1]$ implies $\mathrm{supp}(\mu_V)%
\subseteq[-1,1]$.) The following lemma is about this measure.

\begin{lemma}
\label{1} 
\begin{equation*}
\int\limits_{-1}^1\frac{1}{1-\lambda}\mu_V(\mathrm{d}\lambda)<\infty.
\end{equation*}
\end{lemma}

\begin{proof}
Let $f\in L^2$ be a bounded function. From the symmetry one can easily see that
\[\sum_{e\in S}e\mathbb{E}(f\mathbbm{1}\{e\in S_{\omega}\})=\frac{1}{2}\sum_{e\in S}e\mathbb{E}((f-f\circ\tau_e)\mathbbm{1}\{e\in S_{\omega}\}).\]
So, for any fixed vector $a\in\mathbb{R}^2$:
\begin{align}\label{faV}
<f,a\cdot V>&=\mathbb{E}[f(a\cdot V)]
			=\mathbb{E}\left[f\left(a\cdot\left(\frac{1}{4}\sum_{e\in S}e\mathbbm{1}\{e\in S_{\omega}\}\right)\right)\right]\notag\\
			&=\frac{1}{4}\sum_{e\in S}(e\cdot a)\mathbb{E}[f\mathbbm{1}\{e\in S_{\omega}\}]\notag\\
			&=\frac{1}{2}\frac{1}{4}\sum_{e\in S}(e\cdot a)\mathbb{E}[(f-f\circ\tau_e)\mathbbm{1}\{e\in S_{\omega}\}]\notag\\
            &\le\frac{1}{2}\left(\frac{1}{4}\sum_{e\in S}(e\cdot a)^2\mathbb{P}(e\in S_{\omega})\right)^{\frac{1}{2}}
            \left(\frac{1}{4}\sum_{e\in S}\mathbb{E}[(f-f\circ\tau_e)^2\mathbbm{1}\{e\in S_{\omega}\}]\right)^{\frac{1}{2}}
\end{align}
Here the inequality is the consequence of the Cauchy--Schwartz inequality since the mapping
\[(f_1,f_2)\mapsto\frac{1}{4}\sum_{e\in S}\mathbb{E}[f_1f_2]\]
is a scalar product on the functions defined on $\Omega\times S$ and in our case $f_1=(e\cdot a)\mathbbm{1}\{e\in S_{\omega}\}$,
$f_2=(f-f\circ\tau_e)\mathbbm{1}\{e\in S_{\omega}\}$. The first term of the product obtained in \eqref{faV} is a constant multiple of $|a|$,
while the other term (because of the symmetry) can be written in the following form:
\[\left(\frac{2}{4}\sum_{e\in S}\mathbb{E}(f(f-f\circ\tau_e)\mathbbm{1}\{e\in S_{\omega}\})\right)^{\frac{1}{2}}=
\left(2<f,(1-Q)f>\right)^{\frac{1}{2}}.\]
Therefore we obtained that there exists a constant $C<\infty$ with which for any function $f\in L^2$: 
\[|<f,a\cdot V>|^2\le C|a|^2<f,(1-Q)f>\]
If we apply this result for a function of the form $f=a\cdot\varphi(Q)V$ where we write the two coordinate vector in the place of $a$ we get the
following: for every continuous, bounded function $\varphi:[-1,1]\mapsto\mathbb{R}$
\begin{align*}
\left|\int\limits_{-1}^1\varphi(\lambda)\mu_V(\mathrm{d}\lambda)\right|^2&=|<\varphi(Q)V,V>|^2\\
&=|<(1,0)\cdot\varphi(Q)V,(1,0)\cdot V>+<(0,1)\cdot\varphi(Q)V,(0,1)\cdot V>|^2\\
&\le 2|<(1,0)\cdot\varphi(Q)V,(1,0)\cdot V>|^2+2|<(0,1)\cdot\varphi(Q)V,(0,1)\cdot V>|^2\\
&\le 2C[<(1,0)\cdot\varphi(Q)V,(1,0)\cdot(1-Q)\varphi(Q)V>\\
&\quad\ +<(0,1)\cdot\varphi(Q)V,(0,1)\cdot(1-Q)\varphi(Q)V>]\\
&=2C<\varphi(Q)V,(1-Q)\varphi(Q)V>\\
&=2C\int\limits_{-1}^1(1-\lambda)\varphi(\lambda)^2\mu_V(\mathrm{d}\lambda).
\end{align*}
Substituting $\varphi_{\epsilon}(\lambda)=\mathrm{min}(\frac{1}{\epsilon},\frac{1}{1-\lambda})$ with  $\varphi$ and using that
$(1-\lambda)\varphi_{\epsilon}(\lambda)\le 1$, we get
\[\left|\int\limits_{-1}^1\varphi_{\epsilon}(\lambda)\mu_V(\mathrm{d}\lambda)\right|^2\le
2C\int\limits_{-1}^1\varphi_{\epsilon}(\lambda)\mu_V(\mathrm{d}\lambda),\]
therefore
\[\int\limits_{-1}^1\varphi_{\epsilon}(\lambda)\mu_V(\mathrm{d}\lambda)\le 2C.\]
Using the monotone convergence theorem,
\[\int\limits_{-1}^1\frac{1}{1-\lambda}\mu_V(\mathrm{d}\lambda)=
\sup_{\epsilon>0}\int\limits_{-1}^1\varphi_{\epsilon}(\lambda)\mu_V(\mathrm{d}\lambda)\le 2C<\infty.\]
\end{proof}

Our next lemma is about $\psi_{\epsilon}$ defined above.

\begin{lemma}
\label{2} 
\begin{equation*}
\lim_{\epsilon\searrow 0}\epsilon||\psi_{\epsilon}||_2^2=0.
\end{equation*}
Furthermore, if for $e\in S$, we define $G_e^{(\epsilon)}(\omega)=\mathbbm{1}%
\{e\in S_{\omega}\}(\psi_{\epsilon}\circ\tau_e-\psi_{\epsilon})(\omega)$,
then for any $x\in H$ and for any $e\in S$ the following holds: 
\begin{equation*}
\lim_{\epsilon_1,\epsilon_2\searrow 0}||\mathbbm{1}\{x\in
H_{\omega}\}(G_e^{(\epsilon_1)}\circ\tau_x-G_e^{(\epsilon_2)}\circ%
\tau_x)||_2=0.
\end{equation*}
\end{lemma}

\begin{proof}
By the definition of $\psi_{\epsilon}$,
\[\epsilon||\psi_{\epsilon}||_2^2=\int\limits_{-1}^1\frac{\epsilon}{(1+\epsilon-\lambda)^2}\mu_V(\mathrm{d}\lambda).\]
The integrand is dominated by $\frac{1}{1-\lambda}$ which tends to $0$ on $[-1,1)$ as $\e\searrow 0$. Using Lemma \ref{1} and it's consequence
that $\mu_V(\{1\})=0$, the dominated convergence theorem proves the first proposition.

In order to prove the second proposition, notice that by the shift-invariance of $\mathbb{P}$,
\[||\mathbbm{1}\{x\in
H_{\omega}\}(G_e^{(\epsilon_1)}\circ\tau_x-G_e^{(\epsilon_2)}\circ\tau_x)||_2\le||G_e^{(\epsilon_1)}-G_e^{(\epsilon_2)}||_2.\]
Considering the square of the right side and taking the average for the vectors in $S_{\omega}$:
\begin{align*}
\frac{1}{4}\sum_{e\in S}||G_e^{(\epsilon_1)}-G_e^{(\epsilon_2)}||_2^2
&=\frac{1}{4}\sum_{e\in S}||\mathbbm{1}\{x\in S_{\omega}\}((\psi_{\e_1}-\psi_{\e_2})\circ\tau_e-(\psi_{\e_1}-\psi_{\e_2}))||_2^2\\
&=2<\psi_{\e_1}-\psi_{\e_2},(1-Q)(\psi_{\e_1}-\psi_{\e_2})>\\
&=2\int\limits_{-1}^1\frac{(\e_1-\e_2)^2(1-\lambda)}{(1+\e_1-\lambda)^2(1+\e_2-\lambda)^2}\mu_V(\mathrm{d}\lambda).
\end{align*}
The integrand is dominated by $\frac{1}{1-\lambda}$ again and tends to $0$ as $\e_1,\e_2\searrow 0$. Therefore Lemma \ref{1} and the dominated
convergence theorem proves the second proposition as well.
\end{proof}

\begin{proof}[Proof of Theorem \ref{korrektor}]
Let us use the previously introduced notations. By Lemma \ref{2} $G_e^{(\e)}\circ\tau_x$ converges in $L^2$ as $\e\searrow 0$. Let
$G_{x,x+e}=\lim_{\e\searrow 0}G_e^{(\e)}\circ\tau_x$. It is clear that $G_{x,x+e}(\omega)+G_{x+e,x}(\omega)=0$ and generally:
$\sum_{k=0}^nG_{x_k,x_{k+1}}=0$, if $(x_0,\ldots,x_n)$ is a closed loop in $H_{\omega}$. Therefore the following definition makes sense:
\[\chi(x,\omega):=\sum_{k=0}^{n-1}G_{x_k,x_{k+1}}(\omega),\]
where $(x_0,\ldots,x_n)$ is an arbitrary path in $H_{\omega}$, where $x_0=0$ and $x_n=x$. As we mentioned above, this sum is independent from the
path. The claim about the shift-invariance follows from $G_{x,x+e}=G_{0,e}\circ\tau_x$.

By the shift-invariance, to prove the harmonicity of $x\mapsto x+\chi(x,\omega)$ it is enough that almost surely
\[\frac{1}{4}\sum_{e\in S_{\omega}}\chi(0,\omega)-\chi(e,\omega)=V(\omega).\]
Since $\chi(e,\cdot)-\chi(0,\cdot)=G_{0,e}$, the left side of the equation is the limit of the following as $\e\searrow 0$:
\[\frac{1}{4}\sum_{e\in S_{\omega}}\psi_{\e}-\psi_{\e}\circ\tau_e=(1-Q)\psi_{\e}.\]
By the definition of $\psi_{\e}$, we have $(1-Q)\psi_{\e}=-\e\psi_{\e}+V$. From this we get the desired equality by Lemma \ref{2}
($\e\psi_{\e}\rightarrow 0$ in $L^2$).

In order the prove part (3) of the theorem, notice that by the definition of the corrector
\[(\chi(x+e,\omega)-\chi(x,\omega))\mathbbm{1}\{x\in H_{\omega}\}\mathbbm{1}\{e\in S_{\omega}\}\circ\tau_x=G_{x,x+e}(\omega).\]
$G_{x,x+e}$ is the $L^2$-limit of the functions $G_e^{(\e)}\circ\tau_x$ that have $L^2$-norm bounded above by the $L^2$-norm of $G_e^{(\e)}$
$L^2$. This implies (3) with $C=\max_{e\in S}||G_{0,e}||_2$.
\end{proof}

\subsection{Sublinearity of the corrector along the ribbons}

Firstly, we prove the sublinear growth of the corrector along the ribbons.
For a given configuration $\omega$, let us fix one of the principal
directions directed in a particular way. Denote the centers of the tiles
along this principal direction starting at the origin by $%
z_0(\omega)=0,z_1(\omega),z_2(\omega),\ldots$

\begin{thm}
\label{szalagt} For $\mathbb{P}$-almost every $\omega\in\Omega$: 
\begin{equation*}
\lim_{k\rightarrow\infty}\frac{\chi(z_k(\omega),\omega)}{k}=0.
\end{equation*}
\end{thm}

Firstly, we prove the following:

\begin{prop}
\begin{align*}
\mathbb{E}(|\chi(z_1(\cdot),\cdot)|)&<\infty, \\
\mathbb{E}(\chi(z_1(\cdot),\cdot))&=0.
\end{align*}
\end{prop}

\begin{proof}
$\chi(x,\cdot)$ is the $L^2$-limit of the functions $\chi_{\e}(x,\cdot)=\psi_{\e}\circ\tau_x-\psi_{\e}$
(as $\e\searrow 0$) on the set $\{x\in H_\omega\}$. It is clear that
\[|\chi_{\e}(z_1(\omega),\omega)|\le\sum_{e:e\in S}|G_e^{(\e)}(\omega)|.\]
By Theorem \ref{korrektor} $||G_e^{(\e)}||_2<C$ for every $e$ and $\e>0$. This yields $\chi(z_1(\cdot),\cdot)\in L^1$.

$\mathbb{E}(\chi(z_1(\cdot),\cdot))=0$ follows from the fact that $\chi_{\e}(z_1(\cdot),\cdot)$ converges to $\chi(z_1(\cdot),\cdot)$ in $L^2$,
and $\mathbb{E}(\chi_{\e}(z_1(\cdot),\cdot))=\mathbb{E}(\psi_{\e}\circ\tau_{z_1}-\psi_{\e})=0$.
\end{proof}

\begin{proof}[Proof of Theorem \ref{szalagt}]
Let $f(\omega)=\chi(z_1(\omega),\omega)$ and let $\sigma:\Omega\rightarrow\Omega$ be the shift $\omega\mapsto\tau_{z_1(\omega)}(\omega)$.
$\sigma$ is ergodic (\cite{r}, Theorem A) and from the above $f\in L^1$, $\mathbb{E}(f)=0$. Therefore Birkhoff's ergodic theorem and part (1) of
Theorem \ref{korrektor} yields
\[\lim_{k\rightarrow\infty}\frac{\chi(z_k(\omega),\omega)}{k}=\lim_{k\rightarrow\infty}\frac{\sum_{l=0}^{k-1}f\circ\sigma^l(\omega)}{k}=0.\]
\end{proof}

\subsection{Sublinearity on the whole plane}

After proving the sublinear growth of the corrector along the ribbons, we
can do the same on the whole plane.

\begin{thm}
\label{szublin} 
\begin{equation*}
\lim_{n\rightarrow\infty}\max_{\substack{ x\in H_{\omega}  \\ |x|\le n}}%
\frac{\chi(x,\omega)}{n}=0
\end{equation*}
$\mathbb{P}$-almost surely where $|x|$ denotes the graph distance from the
origin.
\end{thm}

Firstly, we introduce a few notations. Consider the two principal directions
directed in a fixed way. In a configuration $\omega$ denote by $%
a_0(\omega)=0,a_1(\omega),a_2(\omega),\ldots$ the centers of the
consecutively neighboring tiles along the first principal direction in the
positive direction, starting at the origin, and by $0,a_{-1}(\omega),a_{-2}(%
\omega),\ldots$ in the negative direction. For the other principal direction
we use the notions $b_0(\omega)=0,b_1(\omega),b_2(\omega),\ldots$ and $%
0,b_{-1}(\omega),b_{-2}(\omega),\ldots$ similarly. Furthermore, if in a
configuration $\omega$ $x,y\in H_{\omega}$ are on the same ribbon, then
denote by $d_{\omega}(x,y)$ the number of steps needed to take on the ribbon
from $x$ to $y$.

\begin{defi}
For given numbers $K>0$, $\epsilon>0$ we call the point $x\in H_{\omega}$ $%
K,\epsilon$-good in a configuration $\omega$, if 
\begin{equation*}
|\chi(x,\omega)-\chi(y,\omega)|<K+\epsilon d_{\omega}(x,y)
\end{equation*}
holds for every $y\in H_{\omega}$ that is on the same ribbon as $x$ (along
one of the principal directions). The set of the $K,\epsilon$-good points in
a configuration $\omega$ is denoted by $\mathcal{G}_{K,\epsilon}(\omega)$.
\end{defi}

There are $10$ different types of rhombuses that occur in a configuration
from $\Omega$, if we distinguish the rhombuses that can be moved into each
other by a nontrivial rotation (and translation). This is because in the
construction of the tiling the $5$ kinds of lines in the pentagrid creates $%
10$ different kinds of intersections and each has a different type of
rhombus corresponding to it. Denote these $10$ types by $R_1,R_2,%
\ldots,R_{10}$. Let $\mathcal{R}_i$ be the event that the rhombus of the
origin is of type $R_i$. From Theorem \ref{szalagt} we have that for every $%
\epsilon>0$ and $i\in\{1,\ldots,10\}$ there exists $K>0$ such that $\mathbb{P%
}(0\in \mathcal{G}_{K,\epsilon},\ \mathcal{R}_i)>0$. Therefore for any given 
$\epsilon$ there is a $K$ which is good for every $i$ in this sense.

The following lemma is about the density of the $K,\epsilon$-good points of
a given kind.

\begin{lemma}
Fix an $i\in\{1,\ldots,10\}$. Let $\epsilon>0$ and $K$ be large enough such
that $\mathbb{P}(0\in \mathcal{G}_{K,\epsilon},\mathcal{R}_i)>0$. For given $%
n\ge 1$ and $\omega\in\Omega$ let $k_0<k_1<\ldots<k_r$ be the indices from $%
[-n,n]$ for which $a_{k_i}\in \mathcal{G}_{K,\epsilon}(\omega)$ and the
rhombus of $a_{k_i}$ is of type $R_i$ (let us call these good indices). Let 
\begin{equation*}
\Delta_n(\omega)=\max_{j=1,\ldots,r}d_{\omega}(a_{k_j},a_{k_{j-1}}).
\end{equation*}
(If there is no such $k_i$, then let $\Delta_n(\omega)=\infty$.) Then almost
surely 
\begin{equation*}
\lim_{n\rightarrow\infty}\frac{\Delta_n}{n}=0.
\end{equation*}
The same is true for the other ribbon through the origin, with $b_i$'s
instead of $a_i$'s.
\end{lemma}

\begin{proof}
Since $\sigma:\Omega\rightarrow\Omega,\ \omega\mapsto\tau_{a_1}(\omega)$ is ergodic with respect to $\mathbb{P}$, Birkhoff's theorem gives
\begin{equation}\label{lim}
\lim_{n\rightarrow\infty}\frac{1}{n+1}\sum_{k=0}^n\mathbbm{1}\{0\in \mathcal{G}_{K,\e},\ \mathcal{R}_i\}\circ\sigma^k=\mathbb{P}(0\in
\mathcal{G}_{K,\e},\ \mathcal{R}_i)=:p
\end{equation}
almost surely. Now suppose that $\frac{\Delta_n}{n}$ does not tend to $0$, so there is a $\delta>0$ for which $\Delta_n>\delta n$ for infinitely
many $n$'s. By \eqref{lim} there exists an $N$ for which if $n>N$ then
\[\frac{1}{n+1}\sum_{k=0}^n\mathbbm{1}\{0\in \mathcal{G}_{K,\e},\ \mathcal{R}_i\}\circ\sigma^k\in I=(p(1-\delta/16),p(1+\delta/16)),\]
and the same is true when $k$ goes to $(-n)$ in the summation. Let us fix a large $K$ for which $(K\delta-2N-1)/2>K\delta/4$, and
$\Delta_K>\delta K$. Then in $[-K,K]$ there is an interval $[L_1,L_2]$ which is disjoint from $[-N,N]$ and has length larger than $K\delta/4$ in
which there is not any good index. We can assume that this is on the positive half line. Then
\[\sum_{k=0}^{L_2}\mathbbm{1}\{0\in \mathcal{G}_{K,\e},\ \mathcal{R}_i\}\circ\sigma^k\in (L_1+1)I\cap (L_2+1)I.\]
But on the right-hand side, we have the empty set since
\[(L_1+1)p(1+\delta/16)<(L_2+1)p(1-\delta/16),\]
as
\[(L_2-L_1)p-p\delta(L_1+L_2+2)/16>K\delta p/4-p\delta 2K/16>0.\]
\end{proof}

\begin{proof}[Proof of Theorem \ref{szublin}]
Let us fix an $\e>0$ and then a $K_0$ such that  $\mathbb{P}(0\in \mathcal{G}_{K,\e},\mathcal{R}_i)>0$ is true for any $i$ when $K\ge K_0$. Let
$\Omega^*\subseteq\Omega$ be the set of those configurations for which the statement of the previous lemma applies in both principal directions
and for which the part (1) of Theorem \ref{korrektor} applies. ($\mathbb{P}(\Omega^*)=1.$) Let us fix an $\omega\in\Omega^*$ and a $K\ge K_0$ for
which $0\in \mathcal{G}_{K,\e}$. (There exists such $K$ by Theorem \ref{szalagt}.) Let the rhombus of the origin be of type $R_o$. Let
$(x_k)_{k\in\mathbb{Z}}$ be the increasing two-sided sequence for which the points $a_{x_k}$ are exactly the $K,\e$-good points of the
corresponding ribbon through the origin of which the rhombuses are of type $R_o$. Let $n_1(\omega)$ be the least integer for which every $n\ge
n_1(\omega)$ satisfies $\Delta_n/n<\e$ where in the definition of $\Delta$ we take into account the rhombuses of type $R_o$. We define the
indices $(y_k)_{k\in\mathbb{Z}}$ similarly for the other ribbon through the origin; $n_2(\omega)$ corresponds to $n_1(\omega)$ in this case. For
positive $n$ let $u(n)$ be the largest integer for which $x_{u(n)}<n$ while $u(-n)$ is the least integer for which $x_{u(-n)}>-n$. We define the
numbers $v(n)$, $v(-n)$ similarly for the $y$'s instead of the $x$'s. Let $n_0\ge n_1,n_2$ be so large that for $n\ge n_0$ $u(n)$ and $v(n)$ are
positive while $u(-n)$ and $v(-n)$ are negative. Let us denote by $H_n$ the centers of tiles in the area enclosed by the ribbons going through
the points $a_{x_{u(n)}},a_{x_{u(-n)}},b_{y_{v(n)}}$ and $b_{y_{v(-n)}}$ but not through the origin. (This area is finite since the bounding
ribbons on opposite sides are parallel to each other.) We will prove that for every $n\ge n_0(\omega)$
\[\max_{x\in H_n}|\chi(x,\omega)|\le 2K+c\e n\]
with some constant $c$.

Let us denote by $\mathbb{G}$ the union of the good ribbons, where we call a ribbon good if it goes through one of the points $a_{x_k}$ or
$b_{y_k}$. (Since the rhombuses are of the same type, every good ribbon runs along one of the two principal directions.) In order to bound the
value of the corrector, assume that $x\in H_n\setminus\mathbb{G}$. Then $x$ lies in a part of the plane, which is enclosed by two parallel good
ribbons. Two neighboring parallel good ribbons are at most $\e n$ steps from each other measured on the ribbon which goes through the origin and
intersects both of them. Therefore, from $x$ we can access a good ribbon in at most $c_1\e n$ steps, where $c_1$ is a universal constant. So,
using the harmonicity of the function $x\mapsto x+\chi(x,\omega)$, the minimum/maximum principle gives
\begin{equation}\label{max}
\max_{x\in H_n\setminus\mathbb{G}}|\chi(x,\omega)|\le c\e n+\max_{x\in H_{n}\cap\mathbb{G}}|\chi(x,\omega)|
\end{equation}
with some universal constant $c$.

In order to bound the value of the corrector on $H_{n}\cap\mathbb{G}$, let us firstly consider a ribbon that goes through the point $b_{y_k}$ but
not through the origin. For all points $x\in H_{n}$ of this ribbon it is true that $|\chi(x,\omega)-\chi(b_{y_k},\omega)|\le K+\e n$, since
$b_{y_k}\in \mathcal{G}_{K,\e}$ and the shift-invariance holds. Using the same argument for the ribbon connecting $b_{y_k}$ and the origin, we get
\[|\chi(x,\omega)|\le 2K+2\e n\]
for every $x\in\mathbb{G}\cap H_{n}$. Using this and the bound from \eqref{max} we get
\[\max_{x\in H_n}|\chi(x,\omega)|\le 2K+(c+2)\e n.\]
$\e$ was arbitrary, therefore this proves the theorem.
\end{proof}

\section{Invariance principle on the corrected Penrose graph}

The next lemma shows the martingale property of the corrected walk.

\begin{lemma}
\label{martingale} Let us fix $\omega\in\Omega$. Given a path of random walk 
$(X_n)_{n\ge 0}$ with law $P_{\omega}$, let 
\begin{equation*}
M_n^{(\omega)}=X_n+\chi(X_n,\omega)\quad n\ge 0.
\end{equation*}
Then $(M_n^{(\omega)})_{n\ge 0}$ is an $L^2$-martingale for the filtration $%
(\sigma(X_0,\ldots,X_n))_{n\ge 0}$. Moreover, conditioned on $X_{k_0}=x$,
the increments $(M_{k+k_0}^{(\omega)}-M_{k_0}^{(\omega)})_{k\ge 0}$ have the
same law as $(M_k^{(\tau_x\omega)})_{k\ge 0}$.
\end{lemma}

The proof is word by word the same as for Lemma 6.1 in \cite{bb}.

The invariance principle for the corrected walk is stated by the following
theorem.

\begin{thm}
Let us fix $\omega\in\Omega$. Let $(\hat{B}_n^{(\omega)}(t):t\ge 0)$ be
defined by 
\begin{equation*}
\hat{B}_n^{(\omega)}(t)=\frac{1}{\sqrt{n}}(M_{\lfloor
tn\rfloor}^{(\omega)}+(tn-\lfloor tn\rfloor)(M_{\lfloor
tn+1\rfloor}^{(\omega)}-M_{\lfloor tn\rfloor}^{(\omega)}))\quad t\ge 0.
\end{equation*}
Then for all $T>0$ and $\mathbb{P}$-almost every $\omega$, the law of $(\hat{%
B}_n^{(\omega)}(t):0\le t\le T)$ on $(C[0,T],\mathcal{W}_T)$ ($\mathcal{W}_T$
is the usual Borel-$\sigma$-algebra) converges to the law of a Brownian
motion $(B_t:0\le t\le T)$ with diffusion matrix $D$, where 
\begin{equation*}
D=\mathbb{E}\left(E_{\omega}\left(M_1^{(\omega)}{M_1^{(\omega)}}%
^T\right)\right),
\end{equation*}
in other words, 
\begin{equation*}
D_{i,j}=\mathbb{E}(E_{\omega}((e_i\cdot M_1^{(\omega)})(e_j\cdot
M_1^{(\omega)}))),
\end{equation*}
where $e_1,e_2$ are the unit coordinate vectors. This matrix $D$ is positive
definite, so the Brownian motion in the limit is non-degenerate.
\end{thm}

\begin{proof}
Without much loss of generality, we can assume that $T=1$. Let $\mathcal{F}_k=\sigma(X_0,\ldots,X_k)$ and let us fix a vector $a\in\mathbb{R}^2$.
We will show that for any choice of $a$ the piecewise linearization of $t\mapsto a\cdot M_{\lfloor tn\rfloor}^{(\omega)}$ scales to a
one-dimensional Brownian motion with variance $t(a\cdot Da)$. Using the Cramér--Wold device (Theorem 2.9.2 of Durrett \cite{durrett}), this
implies the statement of the theorem for the scaling limit.

We will use the Lindeberg--Feller martingale CLT (Theorem 7.7.3 of Durrett \cite{durrett}). For this let us consider the following probability
variable for $m\le n$:
\[V_{n,m}^{(\omega)}(\e)=\frac{1}{n}\sum_{k=0}^mE_{\omega}(|a\cdot(M_{k+1}^{(\omega)}-M_k^{(\omega)})|^2\mathbbm{1}\{|a\cdot(M_{k+1}^{(\omega)}-
M_k^{(\omega)})|\ge\e\sqrt{n}\}|\mathcal{F}_k).\]
In order to apply the CLT we need to verify that for $\mathbb{P}$-almost every $\omega$ the following two properties hold:\\
(1) $V_{n,\lfloor tn\rfloor}^{(\omega)}(0)\rightarrow t(a\cdot Da)$ in $P_{\omega}$-probability for all $t\in[0,1]$.\\
(2) $V_{n,n}^{(\omega)}(\e)\rightarrow 0$ in $P_{\omega}$-probability for all $\e>0$.

Both of these properties will be implied by ergodicity. By the last conclusion of Lemma \ref{martingale}, we may write
\[V_{n,m}^{(\omega)}(\e)=\frac{1}{n}\sum_{k=0}^mf_{\e\sqrt{n}}(\tau_{X_k}\omega),\]
where
\[f_K(\omega)=E_{\omega}([a\cdot M_1^{(\omega)}]^2\mathbbm{1}\{|a\cdot M_1^{(\omega)}|\ge K\}).\]
Now if $\e=0$, ergodicity implies for $\mathbb{P}$-almost every $\omega$ that
\[\lim_{n\rightarrow\infty}V_{n,\lfloor tn\rfloor}^{(\omega)}(0)=t\mathbb{E}(E_{\omega}([a\cdot M_1^{(\omega)}]^2))=t(a\cdot Da).\]
$D_{i,j}$ is finite since the corrector is square integrable.

On the other hand, when $\e>0$, for any positive $K$, we have $f_{\e\sqrt{n}}\le f_K$ once $n$ is sufficiently large. Therefore
$\mathbb{P}$-almost surely
\[\limsup_{n\rightarrow\infty}V_{n,n}^{(\omega)}(\e)\le\mathbb{E}(E_{\omega}([a\cdot m_1^{(\omega)}]^2\mathbbm{1}\{|a\cdot m_1^{(\omega)}|\ge
K\}))\rightarrow 0,\]
as $K\rightarrow\infty$, where to apply the dominated convergence theorem we used that $a\cdot M_1^{(\omega)}\in L^2$. This completes the proof.

The positive definiteness of $D$ can be seen from that the contrary would mean that there is a vector $v\in\mathbb{R}^2\setminus\{{\bf 0}\}$ for
which \[0=v^TDv=\mathbb{E}\left[E_{\omega}\left[\left(v^TM_1^{(\omega)}\right)\left(v^T{M_1^{(\omega)}}\right)^T\right]\right],\] so $v^T
M_1^{(\omega)}=0$ for almost every $\omega\in\Omega$ and for all four possible initial step. But this would imply that the corrected graph is
almost surely a single point or it is located on a line. But this cannot be possible because the sublinearly growing corrector cannot push the
original graph to a line since the graph has points arbitrarily far from this line.
\end{proof}

\section{The elimination of the corrector}

Our final task is to estimate the effect of the corrector made on the path
of the walk.

\begin{proof}[Proof of Theorem \ref{main}]
We prove that $\tilde B_n(t)\Rightarrow B(t)$, where the covariance matrix of the Brownian motion $B$ is $D$ defined above.
Since $M_n^{(\omega)}=X_n+\chi(X_n,\omega)$, it is enough to show that, for $\mathbb{P}$-almost every $\omega$,
\[\max_{1\le k\le n}\frac{|\chi(X_k,\omega)|}{\sqrt{n}}\rightarrow 0\]
in $P_{\omega}$-probability as $n\rightarrow\infty$. By Theorem \ref{szublin} we know that for $\e>0$ there exists a $K=K(\omega)<\infty$ such
that
\[|\chi(x,\omega)|\le K+\e|x|\quad \forall x\in H_{\omega}.\]
If $\e<1/2$, then this implies
\[|\chi(X_k,\omega)|\le 2K+2\e|M_k^{(\omega)}|.\]
The above CLT for $(M_n^{(\omega)})$ tells us that $\max_{k\le n}|M_k^{(\omega)}|/\sqrt{n}$ converges in law to the maximum of a Brownian motion
$B(t)$ over $t\in[0,1]$. Hence, if we denote the probability law of the Brownian motion by $P$, then by the Portmanteau Theorem (Theorem 2.1 of
Billingsley \cite{billingsley}),
\begin{align*}
\limsup_{n\rightarrow\infty}P_{\omega}\left(\max_{k\le n}|\chi(X_k,\omega)|\ge \delta\sqrt{n}\right)
&\le P\left(2K+2\e\max_{0\le t\le 1}|M_k^{(\omega)}|\ge\delta\sqrt{n}\right)\\
&\le P\left(\max_{0\le t\le 1}|B(t)|\ge\frac{\delta}{2\e}\right).
\end{align*}
The right side tends to zero as $\e\searrow 0$ for all $\delta>0$.
\end{proof}

\end{document}